\documentclass[12pt,a4paper]{article}
\usepackage[dvips]{graphicx}
\usepackage[ansinew]{inputenc}
\usepackage{enumerate}
\usepackage[english]{babel}
\usepackage{amsthm,amssymb,amsfonts,amsmath,amsbsy}
\usepackage[ruled,linesnumbered]{algorithm2e}
\usepackage{color}

\newtheorem{theorem}{Theorem}[section]
\newtheorem{lemma}[theorem]{Lemma}
\newtheorem{proposition}[theorem]{Proposition}
\newtheorem{corollary}[theorem]{Corollary}

\newtheorem{conjecture}{Conjecture}

\newtheorem{observation}[theorem]{Observation}

\date{}

\title{New results on metric-locating-dominating sets of graphs}

\author{A. González\thanks{Departamento de Matemática Aplicada I, Universidad de Sevilla.
E-mail address: gonzalezh@us.es. Partially supported by project MTM2014-60127-P.},
C. Hernando\thanks{Departament de Matemàtiques, Universitat Politècnica de Catalunya.
 E-mail addresses: ${\rm \{}$carmen.hernando,merce.mora${\rm \}}$@upc.edu. Partially supported by projects MTM2012-30951/FEDER and Gen. Cat. DGR 2014SGR46.},
M. Mora$^{\dag}$
}

\begin{document}
\maketitle

\begin{abstract}
A dominating set $S$ of a graph is a metric-locating-dominating set if each vertex of the graph is
 uniquely distinguished by its distances from the elements of $S$,
 and the minimum cardinality of such a set is called the metric-location-domination number.
   In this paper, we undertake a study that,
  in general graphs and specific families, relates
  metric-locating-dominating sets to other special sets: 
  resolving sets, dominating sets, locating-dominating sets and doubly resolving sets.
    We first characterize classes of trees according to certain relationships between
  their metric-location-domination number and their
  metric dimension and domination number.
   Then, we show different methods to transform metric-locating-dominating sets
        into 
 locating-dominating sets and doubly resolving sets.
  Our methods 
  produce new 
   bounds
 on  the minimum cardinalities of all those sets,
 some of them involving parameters that have not been related so far.

\vspace{0.25cm}
\end{abstract}

\section{Introduction and preliminaries}

Let $G=(V(G),E(G))$ be a finite, simple, undirected, and connected graph  of  order $n=|V(G)|\geq 2$;
 the {\em distance} $d(u,v)$ between two vertices $u,v\in V(G)$ is the length of a shortest $u$-$v$ path.
  We say that  a subset $S\subseteq V(G)$ is a
 {\em resolving set} 
 of $G$ if for every $x,y\in V(G)$ there is a vertex $u\in S$
such that  $d(u,x)\neq d(u,y)$ (it is said that $S$ {\em  resolves} $\{x,y\}$), and the minimum cardinality of such a set is called the {\em metric dimension} of $G$,
 written as ${\rm dim} (G)$.
 When $S$ is also a {\em dominating set} of $G$ (i.e., every $x\in V(G)\setminus S$ has a neighbor in $S$), then
 $S$ is called a {\em metric-locating-dominating set} (MLD-set for short).
   The {\em metric-location-domination number} (resp., {\em domination number}),
    written as $\gamma_M(G)$ (resp., $\gamma(G)$),  is the minimum cardinality of an MLD-set (resp., dominating set) of $G$.

MLD-sets were introduced in 2004 by Henning and Oellermann~\cite{HenningOellermann}
   combining the usefulness of resolving sets,
 that roughly speaking differentiate the vertices of a graph, 
and
dominating sets, which cover the whole vertex set.  
Resolving sets 
 were defined in the 1970s by Slater~\cite{slater}, and independently by Harary and Melter~\cite{melter},
whereas dominating sets were introduced  in the 1960s by Ore~\cite{ore}.
Both types of sets have received much attention in the literature 
 because of their many and varied applications in other areas;
 for example, resolving sets serve as a tool for robot navigation, and dominating sets
 are  helpful to design and analyze communication networks
  (see~\cite{bailey} and \cite{fod} for more  information on  resolvability and domination).
  Although 
  MLD-sets are 
  hard to handle,
  for entailing the complexity of the  other two concepts,
   they have been studied in several papers,
 for instance \cite{LDcodes,nordhaus-gaddum}, 
 and further  generalized in other works such as  \cite{McCoy,Stephen}.

This paper first focuses on the intrinsic relations between MLD-sets and resolving sets and dominating sets.
Indeed, the corresponding parameters for all those sets satisfy by definition
 \begin{equation}\label{exp:MLD_L_D} \max\{{\rm dim}(G),\gamma(G)\}\leq \gamma_M(G)\leq {\rm dim}(G)+\gamma(G).\end{equation}
We consider here this chain restricted to trees; specifically, we characterize
the trees for which equality occurs in (\ref{exp:MLD_L_D}), thereby continuing the work of Henning and Oellermann~\cite{HenningOellermann} that  characterized the trees
$T$ with $\gamma_M(T)=\gamma(T)$.
 Analog  characterizations of trees in terms of other  related invariants can be found in \cite{blidia,RNG}.

We also compare MLD-sets with other subsets of vertices defined by Slater~\cite{slaterLD} that are naturally connected to them:
  the {\em locating-dominating sets}.
  They
  are dominating sets that distinguish vertices by using neighborhoods instead of distances;
 more formally,
 a {\em locating-dominating set} (LD-set for brevity) of  $G$ is a dominating set $S\subseteq V(G)$ so that $N(x)\cap S\neq N(y)\cap S$
   for   every $x,y\in V(G)\setminus S$. The minimum cardinality of such a set, denoted by $\gamma_L(G)$, is the
   {\em location-domination number} of $G$.
Clearly, an LD-set is an MLD-set, and so it is also a resolving set; consequently, 
\begin{equation}\label{exp:3chain}{\rm dim }(G)\leq   \gamma_M  (G)\leq\gamma_L(G).\end{equation}
Regarding  the relation between $\gamma_M(G)$ and $\gamma_L(G)$, we propose a way  to
   obtain   LD-sets from MLD-sets  which   helps us to
       extend the following result due to  Henning and Oellermann. 
     (See \cite{LDcodes,nordhaus-gaddum} for more properties of chain (\ref{exp:3chain}), and
   \cite{fod} for specific  results  on LD-sets.)
     \begin{theorem}{\rm \cite{HenningOellermann}}\label{th:HO_MLD_LD}
For any tree $T$, it holds that $\gamma_L(T)<~2\gamma_M(T)$.
However, there is no constant $c$ such that $\gamma_L(G)\leq c\gamma_M(G)$ for all graphs $G$.
\end{theorem}

We finally  find relationships between MLD-sets  and other  subsets for which,
 so far as we are aware,
 no direct connection is known: the {\em doubly resolving sets}.
  Cáceres et al.~\cite{cartesian} introduced doubly resolving sets
as a tool for
 computing the metric dimension of Cartesian products of graphs.
 These sets,
that somehow distinguish vertices in two ways by means of distances, are formally defined as follows.
 Two vertices  $u,v\in V(G)$ {\em doubly resolve} a pair $\{x,y\}\subseteq V(G)$ if $d(u,x)-d(u,y)\neq d(v,x)- d(v,y)$.
A set $S\subseteq V(G)$
is a {\em doubly resolving set}
  of $G$ if   every pair $\{x,y\}\subseteq V(G)$ is doubly resolved by two vertices of $S$
  (it is said that $S$ {\em doubly resolves} $\{x,y\}$),
and the minimum cardinality of such a set is denoted by $\psi(G)$.
 Thus,  a doubly resolving set is also a resolving set,  and so
\begin{equation}\label{exp:psi}{\rm dim }(G)\leq \psi(G).\end{equation}

Although it is not straightforward  to deduce  any relation between  $\psi(G)$ and $\gamma_M(G)$ from their definitions,
 we provide here 
  bounds on $\psi(G)$ in terms of $\gamma_M(G)$
 by  generating  doubly resolving sets from MLD-sets.
 We thus  obtain, for specific classes and general graphs, similar chains to  expression (\ref{exp:3chain})  that include  $\psi(G)$.
  (See for example~\cite{doubly,kratica2,kratica4} 
   for more   results on doubly resolving sets and relations with other types of sets.)

The   paper is organized as follows.
In Section~\ref{sec:trees}, we 
   characterize the extremal graphs of
 expression (\ref{exp:MLD_L_D}) restricted to trees.
 We then show in Sections~\ref{sec:MLD_LD} and  \ref{sec:doubly_MLD}
 how to construct LD-sets and doubly resolving sets from MLD-sets in arbitrary graphs and specific families, thus producing bounds on
 the corresponding parameters.
 Specifically, we prove in Section~\ref{sec:MLD_LD} that $\gamma_L(G)\leq \gamma_M^2(G)$ whenever $G$ has no cycles of length 4 or 6
 but, for arbitrary graphs,   any upper bound on $\gamma_L(G)$
in terms of $\gamma_M(G)$ has growth at least exponential; in
Section~\ref{sec:doubly_MLD}  we achieve the bounds $\psi(G)\leq \gamma_M(G)$ for graphs  $G$ with girth at least 5, and
$\psi(G)\leq \gamma_M(G)+\gamma(G)$ for any graph $G$.
We conclude the paper with  some remarks and open  problems  in Section~\ref{sec:CCRR}.

\section{MLD-sets of trees}\label{sec:trees}

Henning and Oellermann~\cite{HenningOellermann}
provided a formula for the metric-location-domination number of trees  and characterized the trees $T$ with $\gamma_M(T)=\gamma(T)$,
 giving both results   in terms of support vertices (see Theorem~\ref{th:HO_trees} below).
 Recall that 
a vertex $u$ of a tree $T$ 
is a {\em support vertex} whenever it is adjacent to some {\em leaf}  (i.e., a vertex of degree 1),
and it is a {\em strong support vertex} if there are two or more leaves adjacent  to  $u$.
We denote by ${\cal S}(T)$ (resp., ${\cal S}'(T)$) the set of support (resp., strong support) vertices of $T$;
  $\ell'(T)$ is the number of leaves  adjacent to a strong support vertex.

\begin{theorem}{\rm \cite{HenningOellermann}}\label{th:HO_trees}
For any tree $T$, the following statements hold:
\begin{enumerate}
\item[(i)] $\gamma_M(T) = \gamma(T) + \ell'(T) -|{\cal S}'(T)|$.
\item[(ii)]  $\gamma_M(T) = \gamma(T)$ if and only if ${\cal S}'(T)=\emptyset$.
\end{enumerate}
\end{theorem}

As the authors observed, any MLD-set 
must contain, for each support vertex $u$,  either
 all the leaves adjacent to $u$ or all but one
of the leaves adjacent to $u$ as well as   vertex $u$. This
  observation leads us to see that 
  $\ell(T)\leq \gamma_M(T)$, where
$\ell(T)$ denotes the total number of leaves of any tree~$T$.
  Hence, since ${\rm dim} (T)<\ell(T)$ (see   \cite{chartrand}), expression~(\ref{exp:MLD_L_D}) now becomes
 \begin{equation}\label{exp:MLD_L_D_trees}
 \max\{\ell(T),\gamma(T)\}\leq \gamma_M(T)\leq {\rm dim}(T)+\gamma(T).
 \end{equation}

This section follows the same spirit as Henning and Oellermann~\cite{HenningOellermann},
who characterized in statement (ii) of Theorem~\ref{th:HO_trees}
 the extremal trees for expression (\ref{exp:MLD_L_D_trees}) with $\gamma_M(T)=\gamma(T)$.
 Indeed, we characterize the remaining extremal cases:  $\gamma_M (T)={\rm dim}(T)+\gamma(T)$ in Theorem~\ref{th:char1},  and
$\gamma_M (T)=\ell(T)$ 
   in Theorem~\ref{th:char2}. To do this,
 we first recall the following terminology  extracted from~\cite{chartrand}.
A vertex $u\in V(T)$ of degree at least 3 is called a {\it major vertex} of $T$, and
a   leaf  $x\in V(T)$ is a {\it terminal vertex} of 
$u$ if
the major vertex closest to $x$   is $u$.
 The {\it terminal degree} of $u$, written as ${\rm ter}(u)$, is the number of its terminal vertices, and $u$ is an
{\it exterior major vertex} of $T$ if it has positive terminal degree;
 we denote by   ${\rm Ex}(T)$ the set of exterior major vertices of $T$.

\begin{theorem}\label{th:char1} Let $T$ be a tree different from a path. Then, the following statements are equivalent:
\begin{enumerate}
\item[(i)] $\gamma_M (T)={\rm dim}(T)+\gamma(T)$.
\item[(ii)] ${\rm dim}(T)=\ell'(T)-|{\cal S}'(T)|$.
\item[(iii)] Every $u\in {\rm Ex}(T)$ with  ${\rm ter}(u)\geq 2$ is the support vertex of each of  its terminal vertices.
\item[(iv)] Any path joining two leaves of $T$ at distance greater than 2 contains at least two major vertices.

\end{enumerate}
\end{theorem}

\begin{proof}
 \noindent({\it i $\Longleftrightarrow$ ii}) This equivalence  is guaranteed by statement~(i) of Theorem~\ref{th:HO_trees}.

\noindent({\it ii $\Longleftrightarrow$ iii}) It is known that any set $S\subseteq V(T)$  composed
by  all but one of the terminal vertices of each
$u\in {\rm Ex}(T)$  is a minimum resolving set of $T$ 
 (see Theorem 5 of \cite{chartrand} and its proof).
 Thus, let $S\subseteq V(T)$ be such a set, and  note
  that, as any strong support vertex $u$   belongs to ${\rm Ex}(T)$, then  $S$ must contain all but one of the leaves adjacent to $u$;
  consequently,  ${\rm dim}(T)\geq \ell'(T)-|{\cal S}'(T)|$.
 Hence, if ${\rm dim}(T)=\ell'(T)-|{\cal S}'(T)|$ then $S$ is only formed by all but one of the
leaves adjacent to each strong
support vertex.

Let $u\in {\rm Ex}(T)$ with  ${\rm ter}(u)\geq 2$, and let $x$ and $y$ be two terminal vertices of $u$.
 On the contrary, let us assume for instance that $x\not\in N(u)$, and let
$v$ be the only neighbor of $u$ in the  $u$-$x$ path.
 Clearly, $u$ is the support vertex of $y$ because
otherwise there is $v'\in N(u)\setminus \{y\}$ in the $u$-$y$ path
but $\{v,v'\}$ would not be resolved by $S$ since no vertex in either the $u$-$x$ path or the
$u$-$y$ path is in $S$ (by construction of set $S$). Further, one of the leaves adjacent to $u$, say $z$, is not in $S$
 but reasoning as before yields that $S$ does not resolve  $\{v,z\}$; a contradiction since $S$ is a resolving set of $T$.
 Therefore,  $u$  is the support vertex of each of  its terminal vertices.

Reciprocally, let us suppose that each $u\in {\rm Ex}(T)$ with ${\rm ter}(u)\geq 2$ is the support vertex
of its terminal vertices, which implies that $u\in {\cal S}'(T)$.
 Thus, any set  $S\subseteq V(T)$  composed by all but one of the leaves adjacent to each  strong
support vertex 
is a resolving set since   it contains all but one of the terminal vertices of each exterior major vertex.
 Moreover, $|S|=\ell'(T)-|{\cal S}'(T)|$ is minimum since we have noticed in the beginning of this proof
 that  ${\rm dim}(T)\geq \ell'(T)-|{\cal S}'(T)| $.
 Therefore, ${\rm dim}(T)=\ell'(S)-|{\cal S}'(T)|$.

\noindent({\it iii $\Longleftrightarrow$ iv})
Let us suppose that any $u\in {\rm Ex}(T)$ with ${\rm ter}(u)\geq 2$ is adjacent to each of its terminal vertices.
Given two
leaves $x$ and $y$  at distance greater than 2, there is at least one major vertex $u$ in the $x$-$y$ path since $G$ is not
isomorphic to a path. Further, this vertex $u$ cannot be unique since otherwise $x$ and $y$ are terminal vertices of $u$ with either  $d(u,x)\geq 2$ or $d(u',y)\geq 2$, that is,
 $u$ is not the support vertex of either $x$ or $y$, which is impossible. Reciprocally, let us assume statement (iv) to hold true, and let
   $u\in {\rm Ex}(T)$ with ${\rm ter}(u)\geq 2$. It is easily seen that no terminal vertex $x$ of $u$ verifies  $d(u,x)\geq 2$ since otherwise
 $d(x,y)\geq 3$ for any other terminal vertex $y$ of $u$, setting $u$ as the only major vertex in the $x$-$y$ path,  which contradicts statement (iv).
\end{proof}

We want to remark that, although the preceding result does not consider paths, it is easy to check that
     $P_3$ is the only path satisfying $\gamma_M(P_n)=\gamma(P_n)+{\rm dim}(P_n)$ since $\gamma_M(P_n)=\gamma (P_n)$ whenever $n\neq 3$ (by statement (i) of Theorem~\ref{th:HO_trees}), and ${\rm dim}(P_n)=1$  for all $n\geq 1$ (see for instance~\cite{chartrand}).

\begin{theorem}\label{th:char2}
Let $T$ be a tree  different from the path $P_2$. Then, the following statements are equivalent:
\begin{enumerate}
\item[(i)] $\gamma_M(T)=\ell(T)$.
\item[(ii)] $\gamma(T)=|{\cal S}(T)|$.
\item[(iii)] For every $u\in V(T)$, there exists a leaf at
distance at most 2 from $u$.
\end{enumerate}
\end{theorem}

\begin{proof}

\noindent ({\it i $\Longleftrightarrow$ ii}) By statement (i) of Theorem~\ref{th:HO_trees}, $\gamma_M(T)=\ell(T)$ is equivalent to $\gamma(T)=\ell(T)-\ell'(T)+|{\cal S}'(T)|$
but $\ell(T)-\ell'(T)$ is the number of non-strong support vertices, i.e. $|{\cal S}(T)|-|{\cal S}'(T)|$,
which gives $\gamma(T)=|{\cal S}(T)|$.

\noindent ({\it ii $\Longleftrightarrow$ iii})   Let $S\subseteq V(T)$ be a minimum dominating set of $T$.
Observe that, when replacing any leaf of $S$ by its corresponding support vertex, the resulting set is still a dominating set  without leaves
and containing each support vertex of $T$ (since all leaves must be dominated by   vertices of $S$).
 Thus, let us assume ${\cal S}(T)\subseteq S$, which implies that $\gamma(T)\geq |{\cal S}(T)|$ since $S$ has minimum cardinality.
  Therefore,
 $\gamma(T)=|{\cal S}(T)|$ if and only if ${\cal S}(T)$ is a minimum dominating set of $T$, i.e., every $u\in V(T)$
 is either in ${\cal S}(T)$ or has a neighbor in ${\cal S}(T)$. Equivalently, $u$ is either
 a support vertex ($d(u,x)=1$ for some leaf $x\in N(u)$) or a leaf ($d(u,u)=0)$
  or adjacent to a support vertex, say $v$ ($d(u,x)=2$ for some leaf $x\in N(v)$).
\end{proof}

\section{MLD-sets versus LD-sets}\label{sec:MLD_LD}

In view  of the relationship $\gamma_M(G)\leq \gamma_L(G)$  given in expression (\ref{exp:3chain}),
it is natural to look for 
 upper bounds on $\gamma_L(G)$ in terms of $\gamma_M(G)$ as
Henning and Oellermann~\cite{HenningOellermann} did in Theorem~\ref{th:HO_MLD_LD}.
   In this section, we extend this result by first providing
     a wide class of graphs $G$ (that includes trees) whose
      MLD-sets can be transformed into LD-sets; this leads us to the upper bound   $\gamma_L(G)\leq \gamma_M^2(G)$.
     We also  prove, for arbitrary graphs $G$, that broadly speaking  any upper bound
   on $\gamma_L(G)$   as a function  of $\gamma_M(G)$ has   growth at least exponential.

Let $G$ be a graph not having the cycles $C_4$ or $C_6$ as a subgraph, and let $S$ be any subset of $V(G)$.
 We   assign to every pair $u,v\in S$ a set of vertices $\pi(u,v)$ given by $\{u',v'\}$ whenever
 there exists a $u$-$v$ path  $(u,u',v',v)$ (that is unique because of the $C_4$- and $C_6$-free condition),
 and $\emptyset$ otherwise.
  Let $\pi(S)=\bigcup_{u,v\in S}\pi(u,v)$.

\begin{proposition} Let $G$ be a graph not containing $C_4$ or $C_6$ as a subgraph.
For every MLD-set $S\subseteq V(G)$, the set  $S\cup \pi (S)$ is an LD-set of $G$.
Consequently, $\gamma_L(G)\leq \gamma_M^2(G)$.
\end{proposition}

\begin{proof}
  Let $\overline S=S\cup \pi(S)$.
 We  need to check that $N(x)\cap  \overline S\neq N(y)\cap \overline S$
 for every $x,y\in V(G)\setminus \overline S$ (since $S\subseteq \overline S$ is a dominating set of $G$).
 On the contrary, let us assume the existence of two vertices $x,y\in V(G)\setminus \overline S$ so that  $N(x)\cap  \overline S= N(y)\cap \overline S$.
Since $S$ is a dominating set, then there are $u,v\in S$ such that $x\in N(u)$ and $y\in N(v)$.
If $u\neq v$ then either $x\not\in N(v)$ or $y\not\in N(u)$ (otherwise $(u,x,v,y)$ would be a cycle on 4 vertices of $G$, which is impossible),
 and so $N(x)\cap  \overline S \neq N(y)\cap \overline S$.
 Hence, $u=v$. Moreover,
 since the existence of a vertex $v\in N(x)\cap N(y)$ different from $u$ produces the cycle $(u,x,v,y)$, which  cannot exist,
 then  we have that $N(x)\cap  N(y)=\{u\}$, and so $N(x)\cap S=N(y)\cap S=\{u\}$.

Let $z\in V(G)\setminus S$ be a neighbor of either $x$ or $y$
(that exists because otherwise
$N(x)=N(y)=\{u\}$ and so the pair $\{x,y\}$ is 
 not resolved by $S$, which is impossible since $S$ is a resolving set).
 Assuming without loss of generality  $z\in N(x)$, we have that
$z\not\in N(y)$  since we have seen that $N(x)\cap N(y)=\{u\}$. On the other hand,
there is a vertex $u'\in S$ dominating $z$. If $u'\neq u$ then $\pi(u,u')=\{x,z\}\subseteq \pi(S)$; a contradiction since $x\not\in \pi(S)$.
Therefore, $u=u'$ and $N(z)\cap S=\{u\}$. 

Let  $z'\in V(G)\setminus S$ be
such that either $z'\in N(x)\setminus N(z)$ or $z'\in N(z)\setminus N(x)$
  (which exists since otherwise
$N[x]=N[z]=\{x,z,u\}$ and so the pair $\{x,z\}$ 
 is  not resolved by $S$). 
  If $z'\in N(x)\setminus N(z)$ (analogous for the case $z'\in N(z)\setminus N(x)$) then there is $u''\in S$ dominating $z'$,
different from $u$ since otherwise we could obtain the cycle $(u,z',x,z)$. 
Hence,
$\pi(u,u'')=\{x,z'\}\subseteq \pi(S)$, a contradiction with $x\not\in \overline S=S\cup \pi(S)$.

We have thus proved that  $\overline S$ is an LD-set of $G$. To complete the proof,
observe that, for any pair $u,v\in S$, the set $\pi(u,v)$ may intersect either
$S$ or $\pi(u',v')$ for another pair $u',v'\in S$. Consequently,
$|\pi (S)|\leq 2\binom{|S|}{2}$ and so $|\overline S|\leq |S|+|\pi(S)|\leq |S|^2$.
 Therefore, choosing $S$ with minimum cardinality yields $\gamma_L(G)\leq |\overline S|\leq |S|^2= \gamma_M^2(G)$, as required.
\end{proof}

 Henning and Oellermann \cite{HenningOellermann}
showed that there is no linear upper bound on  $\gamma_L(G)$ in terms of $\gamma_M(G)$ by
    building up an appropriate family of graphs
    $G$ with $\gamma_L(G)>c\gamma_M(G)$ for any constant  $c$.
      However, their construction satisfied $\gamma_L (G)<\gamma^2_M (G)$, and so to extend their result to
       other polynomial orders   a new family of graphs is required.
                 We next provide such a family of graphs in the proof of the 
                 following theorem that, in particular, shows that there is no polynomial upper bound on
$\gamma_L(G)$ depending on $\gamma_M(G)$.

\begin{theorem}\label{th:Phi} Let $\Phi:\mathbb{R}\longrightarrow \mathbb{R}$ be a function such that $\gamma_L(G)\leq \Phi(\gamma_M(G))$ for every graph $G$. Then,
$\Phi(x)\geq 2^{x-2}-1$ for all $x\geq 0$.
\end{theorem}

\begin{proof}

To prove the result,
 we  construct a family of graphs $G_s$ such that $\gamma_L(G_s) \geq 2^{\gamma_M(G_s)-2}-1$ as follows.
  For a positive integer $s$,  let $A=\{a_i: i=0,...,2^{s+1}-1\}$, $B=\{b_i: i=0,...,2^{s+1}-1\}$
and $C=\{c_i: i=0,...,s\}$.
The graph $G_s$ has vertex set $V(G_s)=\{p\}\cup A\cup B\cup C$ and
edge set given by the pairs 
 $pa_i$ and $a_ib_i$ for every $i\in\{0,\ldots,2^{s+1}-1\}$, and
$b_ic_j$ whenever the binary representation of $i$ has a 1 in its $j$-th position
 (Figure~\ref{fig:G2} illustrates the case for $s=2$).

\begin{figure}[htb]
\begin{center}
\includegraphics[height=4.85cm]{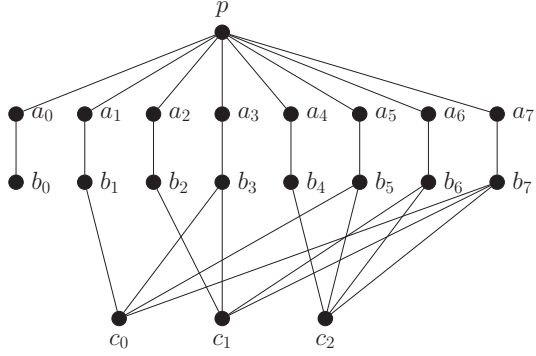}
\caption{The graph $G_2$.}\label{fig:G2}
\end{center}
\end{figure}

It is easy to check that the set $\{C\}\cup\{p\}\cup\{b_0\}$ is  an MLD-set of $G_s$, which implies that
\begin{equation}\label{eq:G_t1}
\gamma_M(G_s)\leq s+3.
\end{equation}
On the other hand,   each LD-set $S$ of $G_s$ must contain,
for each $i\in\{0,\ldots,2^{s+1}-1\}$ but at most one, either vertex $a_i$ or vertex $b_i$.
Indeed, if  there exist $i$ and $j$ such that $a_i,b_i,a_j,b_j\not\in S$ then
$$N(a_i)\cap S = N(a_j) \cap S = \left\{ \begin{array}{lcc}
                      \{p\}  &  if & p\in S \\
                     \\ \emptyset&     & otherwise
             \end{array}
   \right.$$
But this is impossible since $S$ is an LD-set of $G$. Thus,

\begin{equation}\label{eq:G_t2}
\gamma_L(G_s)\geq 2^{s+1}-1.
\end{equation}

Therefore, combining inequalities (\ref{eq:G_t1}) and (\ref{eq:G_t2}) yields $\gamma_L(G_s)\geq 2^{\gamma_M(G_s)-2}-1$,
which gives the result.
\end{proof}

\section{Doubly resolving sets from MLD-sets}\label{sec:doubly_MLD}

In this section, we show how useful MLD-sets can be for constructing doubly resolving sets of graphs.
Indeed, we design a method that, given an MLD-set   of a graph $G$ with $g(G)\geq 5$, 
 produces a doubly resolving set of the same size
 (recall that ${\rm g}(G)$ is the {\it girth} of $G$, i.e., the length of a shortest cycle of $G$);
  when $G$ is any graph, our method also implies the use of dominating sets. In both cases
 we obtain  bounds that involve $\psi(G)$ and $\gamma_M(G)$ (Proposition~\ref{prop:g5} and Theorem~\ref{th:pgg}), giving rise to
  chains similar  to   expression (\ref{exp:3chain}) but including the invariant $\psi(G)$ (Corollaries~\ref{coro:4chain}
 and \ref{coro:2MLD}).
We start with the following two lemmas that are the key     to relate MLD-sets to doubly resolving sets.

\begin{lemma} \label{etadoubly}
Let $G$ be a graph and let $S$ be an MLD-set of $G$. Then, every pair $x,y\in V(G)\setminus S$ is doubly resolved by
$S$.
\end{lemma}

\begin{proof}
Given any two vertices $x,y\in V(G)\setminus S$, we shall prove that there exist $u,v\in S$ such that $d(u,x)-d(u,y)\neq d(v,x)-d(v,y)$.
 Indeed, let $u',v'\in S$ such that $x\in N(u')$ and $y\in N(v')$, which exist since $S$ is a dominating set of $G$.
 We distinguish two cases:
\begin{enumerate}
\item $u'=v'$. We have $d(u',x)-d(u',y)=1-1=0$ and, since $S$ is a resolving set, there is a vertex $w\in S\setminus\{u'\}$
 such that $d(w,x)\neq d(w,y)$, which implies that $d(w,x) - d(w,y) \neq 0$. Therefore,
  we  set $\{u,v\}=\{u',w\}$.
\item $u'\neq v'$. We can assume that $x\not\in N(v')$ and $y\not\in N(u')$ (otherwise we proceed as in the previous case) and so
 $d(u',y),d(v',x)>1$. Hence, $d(u',x)-d(u',y)<0$ and $d(v',x)-d(v',y)>0$, so 
 we can take $\{u,v\}=\{u',v'\}$.
\end{enumerate}
\end{proof}

\begin{lemma} \label{lem:unique}
Let $S$ be an MLD-set of a graph $G$, and let $u\in S$ and $x\in V(G)\setminus S$  such that $\{u,x\}$ is
not doubly resolved by  $S$. Then, $N(x)\cap S=\{u\}$. Furthermore,
 $x$ is the only vertex of $V(G)\setminus S$ so that $\{u,x\}$ is not doubly resolved by $S$.
\end{lemma}

\begin{proof}

 First, we prove that $N(x)\cap S=\{u\}$.
 Observe that, for every $v\in S$, 
\begin{equation}\label{eq:v}
d(u,u)-d(u,x)=d(v,u)-d(v,x)
\end{equation}
 since pair
$\{u,x\}$ is not doubly resolved by $S$  (in particular by $u$ and $v$). As $S$ is a dominating set,   there is some vertex $v^*\in S$
with $x\in N(v^*)$, and so setting $v=v^*$ in  (\ref{eq:v}) yields $-d(u,x)=d(v^*,u)-1$.
Necessarily, $d(v^*,u)=0$ and $d(u,x)=1$, which implies that
$v^*=u$ and (\ref{eq:v}) becomes
\begin{equation}\label{eq:vv}
d(v,x)=d(v,u)+1
\end{equation}
 for each $v\in S$. Hence, $d(v,x)>1$ whenever $v\in S\setminus \{u\}$, and so
$N(x)\cap S=\{u\}$.

Now, we show that there is no other vertex $x'\in V(G)\setminus S$ different from $x$ so that $\{u,x'\}$ is not
doubly resolved by $S$. Let us assume on the contrary   the existence of such a vertex $x'$.
 Reasoning as above with vertex $x$, we easily get
  $d(v,x')=d(v,u)+1$ for any $v\in S$, which combined with (\ref{eq:vv})
gives $d(v,x)=d(v,x')$; a contradiction since $S$ is a resolving set of $G$.
\end{proof}

\begin{observation}\label{obs:uvuv}
For any subset of vertices $S$ of a graph $G$, it is obvious that any pair $\{u,v\}\subseteq  S$ is doubly resolved by $u$ and $v$.
\end{observation}

Regarding  Lemmas~\ref{etadoubly} and \ref{lem:unique} and Observation~\ref{obs:uvuv}, it is natural to ask
whether MLD-sets  $S$ doubly resolve   pairs $\{u,x\}$ with $u\in S$ and $x\in V(G)\setminus S$, thus
implying that MLD-sets would be doubly resolving sets (and so $\psi(G)\leq \gamma(G)$).
Unfortunately,   this is not true in general as graph $H_t$ depicted in Figure~\ref{fig:flower} shows because the set $\{a_1,\ldots,a_t\}$ is an MLD-set of $H_t$ but it does not doubly resolve any pair $\{a_i,c_i\}$; also, $\psi(H_t)=2t=2\gamma_M(G)$.
Furthermore, even adding the extra condition ${\rm g}(G)\geq 5$,
 MLD-sets are not necessarily  doubly resolving sets (see the graph $H'_t$ of Figure~\ref{fig:comb}). 
 However, for this class of graphs, we next describe how to   modify the elements of any MLD-set   to obtain a doubly resolving set,  thereby producing the bound $\psi(G)\leq \gamma_M(G)$.

\begin{figure}[htb]
\begin{center}
\includegraphics[height=2.9cm]{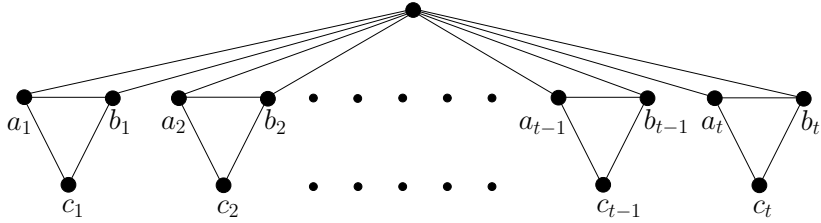}
\caption{The graph $H_t$, with $t\geq 2$, for which $\{a_1,\ldots,a_t\}$ and $\{b_1,\ldots,b_t,c_1,\ldots,c_t\}$ are, respectively, a minimum MLD-set and a minimum doubly resolving set. }\label{fig:flower}
\end{center}
\end{figure}

\begin{figure}[htb]
\begin{center}
\includegraphics[height=1.95cm]{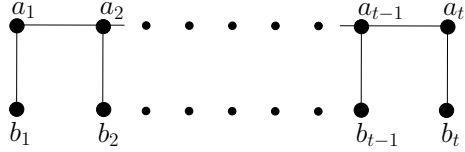}
\caption{The graph $H'_t$, with $t\geq 1$, where $\{a_1,\ldots,a_t\}$ is a minimum  MLD-set but it does
not doubly resolve any pair
$\{a_i,b_i\}$.}\label{fig:comb}
\end{center}
\end{figure}

 Let $S$ be an MLD-set of a graph $G$ with ${\rm g}(G)\geq 5$,
and observe that every $u\in S$ has at most one neighbor   of degree 1 in $V(G)\setminus S$ (otherwise $S$ would not be a resolving set).
Thus, let $\overline u$ be such a neighbor if it exists, and $u$ otherwise. 
Note that, by construction, $\overline u\neq\overline v$ for any two different vertices $u,v\in S$.
The following result proves that changing each $u\in S$ to $\overline u$ yields a doubly resolving set of $G$ whenever $g(G)\geq 5$.

\begin{proposition} \label{prop:g5} Let $G$ be a graph 
with ${\rm g}(G)\geq 5$ and $G\not\cong P_2$. For any MLD-set  $S\subseteq V(G)$,
 the set $\overline S=\{\overline u:u\in S\}$ is a doubly resolving set of $G$. Consequently, $$\psi(G)\leq \gamma_M(G)$$ and this bound is tight.
\end{proposition}

\begin{proof}
We begin by noticing that, given  a vertex $u\in S$, we have that
   \begin{equation}\label{eq:d+1} d(\overline u,x)-d(\overline u,y)=d(u,x)-d(u,y)
   \end{equation}
  for any two vertices  $x,y\neq \overline u$
  (because if $\overline u\neq u$ then $N(\overline u)=\{u\}$, and so  $d(\overline u,x)-d(\overline u,y)=(d(u,x)+1)-(d(u,y)+1)$).
To prove that $\overline S$ is a doubly resolving set, we first show that $\overline S$ doubly resolves at least the same pairs as $S$.
    Indeed, let $\{x,y\}\subseteq V(G)$ be a pair that
   is   doubly resolved by $S$, i.e., there exist $u,v\in S$ such that $d(u,x)-d(u,y)\neq d(v,x)-d(v,y)$. We shall see that $\overline u$ and $\overline v$ (which must be different) doubly resolve $\{x,y\}$.  Clearly, if $x,y\not \in\{\overline u,\overline v\}$ then $\overline u$ and $\overline v$ doubly resolve $\{x,y\}$, by (\ref{eq:d+1});
     otherwise, either $\{x,y\}=\{\overline u, \overline v\}$ (and so $\overline u$ and $\overline v$ doubly resolve $\{x,y\}$,
      by Observation~\ref{obs:uvuv})
      or $|\{\overline u,\overline v\}\cap\{x,y\}|=1$. Thus, let us assume without loss of generality that
      $x=\overline u$ and $y\neq \overline v$. We distinguish two cases.

      {\em Case 1.} $u=\overline u$: We have $d(\overline u,x)-d(\overline u,y)=d(u,x)-d(u,y)$ and, by   (\ref{eq:d+1}),  $d(\overline v,x)-d(\overline v,y)= d(v,x)-d(v,y)$.
           But $d(u,x)-d(u,y)\neq d(v,x)-d(v,y)$ by assumption,  which implies that $\overline u$ and $\overline v$ doubly resolves $\{x,y\}$.

      {\em Case 2.} $u\neq \overline u$: By the triangle inequality,  $d(\overline v,y)\leq d(\overline v,u)+d(u,y) \leq(d(\overline v,\overline u)-1)+(d(\overline u,y)-1)$, or  equivalently $ -d(\overline u,y)\leq d(\overline v,x)-d(\overline v,y)-2$ since $\overline u=x$.
     Thus,  $d(\overline u,x)-d(\overline u,y)\leq d(\overline v,x)-d(\overline v,y)-2<d(\overline v,x)-d(\overline v,y)$, and then  $\{x,y\}$ is doubly resolved by $\overline u$ and $\overline v$.

Now, we show that $\overline S$ also doubly resolves the pairs $\{x,y\}$ not being doubly resolved by $S$.
 By Lemmas~\ref{etadoubly} and \ref{lem:unique} and Observation~\ref{obs:uvuv}, we can assume   that
 $x \in S$,  $y\in V(G)\setminus S$ and  $N(y)\cap S=\{x\}$, being $y$ the only vertex in
 $V(G)\setminus S$ so that  $\{x,y\}$ is not  doubly resolved by $S$. 
  Furthermore, we can prove that $N(y)=\{x\}$ (this is shown below),
which implies that  $\overline x=y\in\overline S$. Thus,
pair $\{x,y\}$ is doubly resolved by $y\in\overline S$ 
and any vertex $v\in\overline S\setminus \{y\}$ since
 $d(y,x)-d(y,y)=1\neq -1 =d(v,x)-d(v,y)$ (note that such a vertex $v$ exists since $|S|\geq \gamma_M(G)\geq 2$ as
 $G\not\cong P_2$; see for instance \cite{LDcodes}). Therefore, we have   proved that   $\overline S$ is a doubly resolving set of $G$ of cardinality $|S|$, which gives
 $\psi(G)\leq \gamma_M(G)$ by choosing $S$ of minimum cardinality.
 Moreover, this bound is tight
 because the graph $H'_t$ of Figure~\ref{fig:comb} satisfies $\psi(H'_t)=\gamma_M(H'_t)=t$
 (it is easy to see that $\{b_1,\ldots,b_t\}$ is the unique minimum doubly resolving set of $H_t'$; see \cite{cartesian} for details).

 To finish the proof, it only remains to check that any pair $\{x,y\}$ that is not doubly
  resolved by $S$ satisfies $N(y)=\{x\}$. On the contrary, let us suppose the existence of a vertex   $z\in N(y)\setminus \{x\}$, which is not in 
  $S$ since $N(y)\cap S=\{x\}$.
 As $S$ is an MLD-set, there is a vertex $w\in S\cap N(z)$,  which must be different from $x$ since $g(G)\geq 5$.
 For the same reason, we have that 
 $d(w,x)\geq 2$, and also that  $d(w,y)=2$. Thus,
 $d(w,x)-d(w,y)\geq 0$ but $d(x,x)-d(x,y)=-1$, so $w,x\in S$ doubly resolve $\{x,y\}$; a contradiction.
 \end{proof}

This last result, together with expression (\ref{exp:psi}), allows us to place $\psi(G)$
into  the chain of expression  (\ref{exp:3chain}) as the following corollary shows.

\begin{corollary} \label{coro:4chain} Let $G$ be a graph with $g(G)\geq 5$. Then,
$${\rm dim}(G)\leq \psi(G)\leq \gamma_M(G)\leq \gamma_L(G).$$
\end{corollary}

Now, we provide a   bound on $\psi(G)$ 
  for arbitrary  graphs $G$. 
 To do this, we follow  a similar process  than in the proof of Proposition~\ref{prop:g5}
 but using also dominating sets.

\begin{theorem}\label{th:pgg} For every graph $G$, it holds that $$\psi(G)\leq \gamma_M(G) + \gamma(G)$$ and this bound is tight.
\end{theorem}

\begin{proof}
Let $S_1$ and $S_2$ be a minimum $MLD$-set and a minimum dominating set, respectively, of $G$. Also,
let $\{x,y\}\subseteq V(G)$ be a
pair that is not doubly resolved by $S=S_1\cup S_2$. Since set $S$ is in particular  an MLD-set,
by Lemmas~\ref{etadoubly} and \ref{lem:unique} and Observation~\ref{obs:uvuv}, we can assume
 $x \in S$, $y$ to be the only vertex of $V(G)\setminus S$ so that $\{x,y\}$ is not doubly resolved by $S$, and
$N(y)\cap S=\{x\}$.
Furthermore, $x\in S_1\cap S_2$ because  $x\in S_1\setminus  S_2$ (analogous for $x\in S_2\setminus S_1$) implies that
      there is $u\in S_2$ dominating $y$  since $S_2$ is a dominating set,  which contradicts $N(y)\cap S= \{x\}$.

Let $S'$ be the set of vertices $y\in V(G)\setminus S$ so that $\{x,y\}$  is  not  doubly resolved  by $S$  for some $x\in S_1\cap S_2$
(note that $| S'|\leq |S_1\cap S_2|$ by the uniqueness of each vertex $y\in V(G)\setminus S$).
 Clearly, $S\cup  S'$ is a doubly resolving set of $G$ and has cardinality
 $|S|+| S'|\leq |S_1\cup S_2|+|S_1\cap S_2|= \gamma_M(G)+\gamma(G)$, which yields the expected bound.
 To prove tightness, we consider the   graph $H_t$ of Figure~\ref{fig:flower}
which verifies $\psi(H_t)=2t$ and $\gamma_M(H_t)=\gamma(H_t)=t$ for each $t\geq 2$.
\end{proof}

 Combining Theorem~\ref{th:pgg} and the fact that $\gamma(G)\leq \gamma_M(G)$ for
 any graph $G$, 
  we achieve  the following chain that is similar
 to expression (\ref{exp:3chain}) and includes $\psi(G)$.

\begin{corollary} \label{coro:2MLD} For every graph $G$, it holds that
$${\rm dim}(G)\leq \psi(G)\leq 2\gamma_M(G)\leq 2 \gamma_L(G).$$
\end{corollary}

We remark that tightness in the bound $\psi(G)\leq 2\gamma_M(G)$ is guaranteed by the graph $H_t$ of Figure~\ref{fig:flower}.

Finally, we propose the following conjecture that is supported by Proposition~\ref{prop:g5} since $  \gamma_M(G)\leq {\rm dim}(G)+\gamma(G)$.

\begin{conjecture}\label{conj:dim+gamma} For every graph $G$, it holds that $$\psi(G)\leq {\rm dim}(G)+\gamma(G).$$
\end{conjecture}

\section{Concluding remarks and open questions}\label{sec:CCRR}

In this paper, we have first characterized the trees $T$ in the cases   $\gamma_M(T)={\rm dim}(T)+\gamma(T)$  and
 $\gamma_M(T)=\ell (T)$.
  We   have then shown how to obtain LD-sets from MLD-sets of
  graphs $G$ without $C_4$ or $C_6$, giving rise to
 the polynomial bound $\gamma_L(G)\leq \gamma_M^2(G)$. For arbitrary graphs $G$,
 we have proved that  any upper bound on $\gamma_L(G)$ as a function of  $\gamma_M(G)$
has growth at least exponential.  Finally,
we have constructed doubly resolving sets from MLD-sets in order to show
that  $\psi(G)\leq \gamma_M(G)$ whenever ${\rm g}(G)\geq 5$, and 
$\psi(G)\leq \gamma_M(G)+\gamma(G)$ for any graph $G$.

It would be interesting
to characterize the trees $T$ with ${\rm dim}(T)=\gamma(T)$.
   Also, we could find new polynomial upper bounds on $\gamma_L(G)$ in terms of $\gamma_M(G)$
 for other specific families of graphs.
  For arbitrary graphs,
 Theorem~\ref{th:Phi} could be improved 
  by providing either a new construction (better than the graph $G_s$) or an 
   upper bound on $\gamma_L(G)$ depending on  $\gamma_M(G)$.
  Concerning $\psi(G)$, it would be of interest to
 find other classes of graphs $G$ with  $\psi(G)\leq \gamma_M(G)$, as well as to
  settle Conjecture~\ref{conj:dim+gamma}.

\bibliographystyle{plain}
\bibliography{MLD_biblio}

\begin{thebibliography}{10}

\bibitem{bailey}
R.~F. Bailey and P.~J. Cameron.
\newblock Base size, metric dimension and other invariants of groups and
  graphs.
\newblock {\em Bull. Lond. Math. Soc.}, 43(2):209--242, 2011.

\bibitem{blidia}
M.~Blidia, M.~Chellali, F.~Maffray, J.~Moncel, and A.~Semri.
\newblock Locating-domination and identifying codes in trees.
\newblock {\em Australas. J. Combin.}, 39:219--232, 2007.

\bibitem{LDcodes}
J.~C{\'a}ceres, C.~Hernando, M.~Mora, I.~M. Pelayo, and M.~L. Puertas.
\newblock Locating-dominating codes: bounds and extremal cardinalities.
\newblock {\em Appl. Math. Comput.}, 220:38--45, 2013.

\bibitem{cartesian}
J.~C{\'a}ceres, C.~Hernando, M.~Mora, I.~M. Pelayo, M.~L. Puertas, C.~Seara,
  and D.~R. Wood.
\newblock On the metric dimension of {C}artesian products of graphs.
\newblock {\em SIAM J. Discrete Math.}, 21(2):423--441 (electronic), 2007.

\bibitem{chartrand}
G.~Chartrand, L.~Eroh, M.~A. Johnson, and O.~R. Oellermann.
\newblock Resolvability in graphs and the metric dimension of a graph.
\newblock {\em Discrete Appl. Math.}, 105(1-3):99--113, 2000.

\bibitem{RNG}
D.~Garijo, A.~Gonz{\'a}lez, and A.~M{\'a}rquez.
\newblock The resolving number of a graph.
\newblock 2013.
\newblock arXiv:1309.0252v1.

\bibitem{melter}
F.~Harary and R.~A. Melter.
\newblock On the metric dimension of a graph.
\newblock {\em Ars Combinatoria}, 2:191--195, 1976.

\bibitem{fod}
T.~W. Haynes, S.~T. Hedetniemi, and P.~J. Slater.
\newblock {\em Fundamentals of domination in graphs}, volume 208 of {\em
  Monographs and Textbooks in Pure and Applied Mathematics}.
\newblock Marcel Dekker Inc., New York, 1998.

\bibitem{HenningOellermann}
M.~A. Henning and O.~R. Oellermann.
\newblock Metric-locating-dominating sets in graphs.
\newblock {\em Ars Combin.}, 73:129--141, 2004.

\bibitem{nordhaus-gaddum}
C.~Hernando, M.~Mora, and I.~M. Pelayo.
\newblock Nordhaus-{G}addum bounds for locating domination.
\newblock {\em European J. Combin.}, 36:1--6, 2014.

\bibitem{doubly}
J.~Kratica, M.~{\v{C}}angalovi{\'c}, and
  V.~Kova{\v{c}}evi{\'c}-Vuj{\v{c}}i{\'c}.
\newblock Computing minimal doubly resolving sets of graphs.
\newblock {\em Comput. Oper. Res.}, 36(7):2149--2159, 2009.

\bibitem{kratica2}
J.~Kratica, V.~Kova{\v{c}}evi{\'c}-Vuj{\v{c}}i{\'c}, M.~{\v{C}}angalovi{\'c},
  and M.~Stojanovi{\'c}.
\newblock Minimal doubly resolving sets and the strong metric dimension of some
  convex polytopes.
\newblock {\em Appl. Math. Comput.}, 218(19):9790--9801, 2012.

\bibitem{McCoy}
J.~McCoy and M.~A. Henning.
\newblock Locating and paired-dominating sets in graphs.
\newblock {\em Discrete Appl. Math.}, 157(15):3268--3280, 2009.

\bibitem{kratica4}
N.~Mladenovi{\'c}, J.~Kratica, V.~Kova{\v{c}}evi{\'c}-Vuj{\v{c}}i{\'c}, and
  M.~{\v{C}}angalovi{\'c}.
\newblock Variable neighborhood search for metric dimension and minimal doubly
  resolving set problems.
\newblock {\em European J. Oper. Res.}, 220(2):328--337, 2012.

\bibitem{ore}
O.~Ore.
\newblock {\em Theory of graphs}.
\newblock American Mathematical Society Colloquium Publications, Vol. 38.
  American Mathematical Society, Providence, R.I., 1962.

\bibitem{slater}
P.~J. Slater.
\newblock Leaves of trees.
\newblock In {\em Proceedings of the {S}ixth {S}outheastern {C}onference on
  {C}ombinatorics, {G}raph {T}heory, and {C}omputing ({F}lorida {A}tlantic
  {U}niv., {B}oca {R}aton, {F}la., 1975)}, pages 549--559. Congressus
  Numerantium, No. XIV, Winnipeg, Man., 1975. Utilitas Math.

\bibitem{slaterLD}
P.~J. Slater.
\newblock Dominating and reference sets in a graph.
\newblock {\em J. Math. Phys. Sci.}, 22(4):445--455, 1988.

\bibitem{Stephen}
S.~Stephen, B.~Rajan, C.~Grigorious, and A.~William.
\newblock Resolving-power dominating sets.
\newblock {\em Appl. Math. Comput.}, 256:778--785, 2015.

\end{thebibliography}

\end{document}